\documentclass[11pt]{article}
\usepackage{amsfonts}
\usepackage{latexsym,amssymb,amsfonts}
\usepackage{mathrsfs,amsthm,graphicx,color}
 \usepackage{psfrag}
 \usepackage{hyperref}
\usepackage{amsmath, amsbsy}
\usepackage{amsopn, amstext}
\usepackage{amsmath,multicol,amsopn}

\def\sqr#1#2{{\vcenter{\vbox{\hrule height.#2pt
              \hbox{\vrule width.#2pt height#1pt \kern#1pt \vrule width.#2pt}
              \hrule height.#2pt}}}}

\def\dbR{{\mathop{\rm l\negthinspace R}}}

\def\3n{\negthinspace \negthinspace \negthinspace }
\def\2n{\negthinspace \negthinspace }
\def\1n{\negthinspace }

\def\ds{\displaystyle}

\def\dbR{{\mathop{\rm l\negthinspace R}}}
\def\={\buildrel \triangle \over =}

\def\l{\lambda}

\def\ns{\noalign{\ss} }

\def\diag{\hbox{\rm $\,$diag$\,$}}
\def\O{\Omega}

\def\ms{\medskip}

\def\q{\quad}

\def\max{\mathop{\rm max}}
\def\min{\mathop{\rm min}}

\def\pa{\partial}

\def\cds{\cdots}

\def\|{\Big |}
\def\({\Big (}
\def\){\Big )}
\def\[{\Big[}
\def\]{\Big]}
\def\be{\begin{equation}}
\def\bel{\begin{equation}\label}
\def\ee{\end{equation}}
\def\bt{\begin{theorem}}
\def\bcd{\begin{condition}}
\def\ecd{\end{condition}}
\def\et{\end{theorem}}
\def\bc{\begin{corollary}}
\def\ec{\end{corollary}}
\def\bde{\begin{definition}}
\def\ede{\end{definition}}
\def\bl{\begin{lemma}}
\def\el{\end{lemma}}
\def\bp{\begin{proposition}}
\def\ep{\end{proposition}}
\def\br{\begin{remark}}
\def\er{\end{remark}}
\def\ba{\begin{array}}
\def\ea{\end{array}}
\def\ed{\end{document}}
\def\ns{\noalign{\ms}}
\def\ds{\displaystyle}

\def\square#1{\vbox{\hrule\hbox{\vrule height#1%
     \kern#1\vrule}\hrule}}
\def\rectangle#1#2{\vbox{\hrule\hbox{\vrule height#1%
     \kern#2\vrule}\hrule}}

\font\tenbb=msbm10 \font\sevenbb=msbm7
\font\fivebb=msbm5

\newfam\bbfam
\scriptscriptfont\bbfam=\fivebb
\textfont\bbfam=\tenbb
\scriptfont\bbfam=\sevenbb

\newtheorem{lemma}{Lemma}[section]
\newtheorem{remark}{Remark}[section]
\newtheorem{example}{Example}[section]
\newtheorem{theorem}{Theorem}[section]
\newtheorem{corollary}{Corollary}[section]

\newtheorem{definition}{Definition}[section]
\newtheorem{proposition}{Proposition}[section]
\newtheorem{condition}{Condition}[section]

\makeatletter
   
   \@addtoreset{equation}{section}
\makeatother

\begin{document}
\title{\textbf{Some Sufficient Conditions for the Controllability of Wave Equations with Variable Coefficients}}
\date{}
\author{ Yuning Liu\thanks{Faculty of Mathematics,
University of Regensburg, D-93053,
Regensburg, Germany. (liuyuning.math@gmail.com). The author was partially supported by the University of Regensburg.}\quad\quad
 }

\maketitle
\abstract{In this work, we present some easily verifiable sufficient conditions that guarantee the controllability of wave equations with non-constant coefficients. These conditions work as complements for those obtained in \cite{Yao1}. }

\section{Introduction and the Main Results}

Let $T>0$ and $\O\subset \mathbb{R}^n$ be a bounded
domain with a $C^2$ boundary $\pa\O$. Let
$a^{ij}\in C^1(\overline \O)$($i,j=1,\cds,n$)
such that $a^{ij}=a^{ji}$ and
$A\=(a^{ij})_{1\leq i,j\leq n}$ is a uniformly
positive definite matrix. Consider the following
hyperbolic equation:
\begin{equation}\label{system1}
\left\{
\begin{array}{ll}\ds
y_{tt} - \sum_{i,j=1}^n
\big(a^{ij}y_{x_i}\big)_{x_j}=0 &\mbox{ in }
(0,T)\times\O,\\
\ns\ds y=0 &\mbox{ on } (0,T)\times \pa\O,\\
\ns\ds y(0)=y_0,\;y_t(0)=y_1 &\mbox{ on }\O.
\end{array}
\right.
\end{equation}
Here $(y_0,y_1)\in H_0^1(\O)\times L^2(\O)$. In
order to establish the boundary observability
estimate for the equation \eqref{system1} by
multiplier method or Carleman estimate, one needs the following conditions (see
\cite{FYZ} for example):

\begin{condition}\label{con}
There exists a function $d\in C^2(\overline\O)$
such that
\begin{equation}\label{con1}
\sum_{i,j=1}^n\bigg\{\sum_{i',j'=1}^n\[
2a^{ij'}(a^{i'j}d_{x_i'})_{x_{j'}} -
a^{ij}_{x_{j'}}a^{i'j'}d_{x_{i'}}
\]\bigg\}\xi^i\xi^j \geq \mu_0\sum_{i,j=1}^n
a^{ij}\xi_i\xi_j,
\end{equation}
when $(x,\xi_1,\cds,\xi_n)\in\overline \O\times
\mathbb{R}^n$, and such that
and
\begin{equation}\label{con2}
|\nabla d|>0  \q\mbox{ in } \overline{\O}.
\end{equation}
\end{condition}

\begin{remark}\label{rm1}
One can directly verify the following: The condition \eqref{con1}
is equivalent to that the  matrix
\begin{equation}\label{B}
B=(b^{ij})_{1\leq i,j\leq n}\=\bigg(
\sum_{i',j'=1}^n\(
a^{ij'}a^{i'j}d_{x_{i'}x_{j'}} +
\frac{a^{ij'}a_{x_{j'}}^{i'j} +
a^{jj'}a^{i'i}_{x_{j'}}-a^{ij}_{x_{j'}}a^{i'j'}}{2}d_{x_{i'}}
\)\bigg)_{1\leq i,j\leq n}
\end{equation}
is uniformly positive definite.
\end{remark}

The function $d$ verifying (\ref{con1}) and (\ref{con2}) does not exist for some cases. This can be seen from the following example:
\begin{example}\label{ex1}
Let $\O=\{(x,y):\,x^2+y^2\leq 2\}$. Let
$(a^{ij})_{1\leq i,j\leq 2}=\diag(a^{1},a^{2})$
with $a^{1}(x,y)=a^{2}(x,y) = 1+x^2+y^2$. By an indirect proof based on the Geometric
Control Condition given in \cite{BLR}, we can
show that there is no such a function  $d$  that satisfies
\eqref{con1}.
\end{example}

Now, we study the existence of functions $d$ verifying (\ref{con1}) and (\ref{con2}) for
suitable $(a^{ij})_{1\leq i,j\leq n}$.
 We  will focus our studies on the special case where
$A=(a^{ij})_{1\leq i,j\leq
n}=\diag(a^1,\cds,a^n)$, where $a^i\in
C^1(\overline\O)$. From Example \ref{ex1}, we
see that even in this case, the above-mentioned functions $d$ may not exist. Thus, it is interesting to provide certain
easily verifiable condition to ensure the existence of such functions $d$ in the case when $A$ is diagonal.
The main results of this study are as follows:
\begin{theorem}\label{maintheorem}
Let  $A=\diag(a^1,\cds,a^n)$, with $a^i\in
C^1(\overline\O)~(1\leq i\leq n)$, be positive uniformly definite over $\overline{\Omega}$. If there exists $j\in\{1,\cdots,n\}$ such that all the terms of  $\{a^i_{x_j}\}_{1\leq i\leq n;i\neq j}$ remain positive (or negative ) over $\overline{\Omega}$, then there is  a function $d\in C^2(\overline\O)$ verifying Condition \ref{con}.
\end{theorem}
It is worth mentioning that, in the statement of the main theorem, we don't need the structural condition on $a^j_{x_j}$ where $j$ is the fixed index.
Before carrying out the proof, we give two corollaries. The first one corresponds to the case $j=1$:
\begin{corollary}\label{th1}
Let  $A=\diag(a^1,\cds,a^n)$, with $a^i\in
C^1(\overline\O)$, $i=1,\cdots, n$, be positive uniformly definite over $\overline{\Omega}$.
Suppose that
\begin{equation}\label{wang1.5}
a^k_{x_1}>0~(\text{or}~ a^k_{x_1}<0)~\mbox{over}\;\;\overline{\Omega},\;\;\mbox{for}\;\; 2\leq k\leq n,
\end{equation}
 then, there is  a function $d\in C^2(\overline\O)$ verifying Condition \ref{con}.
\end{corollary}

\begin{corollary}\label{th2}
Let $A=\diag(a^1, a^2)$, with $a^1, a^2\in
C^1(\overline\O)$, be positive uniformly definite over $\overline{\Omega}$.
Suppose   that $a^1_{x_2}$ (or $a^2_{x_1}$)  is either
positive or negative over $\overline \Omega$. Then there is a function  $d\in C^2(\overline\O)$ satisfying
Condition \ref{con}.
\end{corollary}

\section{Proof of Main Theorem:}
 \begin{proof}[Proof of Theorem \ref{maintheorem}: case 1]
  We first consider the following case:
  \begin{equation}\label{assump1}
    a^i_{x_j}<0,\quad \text{uniformly over } x\in\overline{\Omega},~\text{for all}~ 1\leq i\leq n~\text{with}~i\neq j.
  \end{equation}
  where $j$ is a fixed index.
  Let  $$\ds
d\triangleq d(x)=e^{\l(c+x_j)}+\sum_{1\leq i\leq n,i\neq j} e^{\l x_i},\;\; x\in \overline{\Omega},$$ where
$c>0$ satisfies that
\begin{equation}\label{newas1}
\min_{x\in\overline\O}\{c+x_j\}\geq 1+
\max_{x\in
\overline\O}\sum_{1\leq i\leq n,i\neq j}|x_i|,
\end{equation}
and
$\l>0$ is a large number will be determined later. Using \eqref{newas1}, one could check that the function $d(x)$ enjoys the following properties:
\begin{itemize}
\item For any $1\leq i\leq n$,
\begin{equation}\label{newproperties3}
  d_{x_ix_i}>0,~d_{x_i}>0,\qquad \text{uniformly for  $x\in\overline{\Omega}$ }.
  \end{equation}
   \item For any $1\leq i\leq n$,
\begin{equation}\label{newproperties1}
  \lim_{\lambda\to+\infty}\frac{d_{x_i}}{d_{x_jx_j}}=0\qquad \text{uniformly for  $x\in\overline{\Omega}$ }.
  \end{equation}
  \item For any $1\leq i\leq n$ with $i\neq j$,
\begin{equation}\label{newproperties2}
  \lim_{\lambda\to+\infty}\frac{d_{x_i}}{d_{x_j}}=0,\quad\lim_{\lambda\to+
  \infty}\frac{d_{x_ix_i}}{d_{x_j}}=0,\quad \text{uniformly for  $x\in\overline{\Omega}$ }.
\end{equation}
 \end{itemize}

From Remark \ref{rm1}, to prove $d$ enjoys
\eqref{con1} for the case
$A=\diag(a^1,\cds,a^n)$, we only need to show
the uniformly positivity of the following
matrix:
\begin{multline}\label{B1}
B=\frac 12\(
a^ia^j_{x_i}d_{x_j} + a^j a_{x_j}^i
d_{x_i} \)_{1\leq i,j\leq n} \\
+\diag\((a^1)^2d_{x_1x_1}-\frac 12 \sum_{k=1}^n a^k a^1_{x_k}
d_{x_k},\cds,(a^n)^2d_{x_nx_n}-\frac 12\sum_{k=1}^n a^k a_{x_k}^n d_{x_k}
\).
\end{multline}

To achieve this goal, we only need to show that
all the leading principal minors of $B$ are
positive. In order to avoid the terrible expansion of the determinant, we shall make full use of the asymptotic behavior with respect to the parameter $\lambda$.
We denote by $e_i$  the $i$-th standard basis of $\mathbb{R}^n$ and by $\{B_i\}_{i=1}^n$ the row vector of $B$. It can be verified that, with a very large $\lambda>0$, the matrix $B$ is uniformly positive definite over $\overline{\O}$  if and only if all the leading principal minors of
 the matrix
 $\tilde{B}(x,\lambda):=  \left(
\begin{array}{cc}
\frac{B_1}{d_{x_j}}\\
\vdots\\
\frac{ B_{j-1}}{d_{x_j}}\\
\frac{ B_j}{d_{x_jx_j}}\\
\frac{ B_{j+1}}{d_{x_j}}\\
\vdots\\
\frac{ B_n}{d_{x_j}}
\end{array}
\right)$
 is uniformly positive over $\overline{\O}$. This later condition is relatively easier to be verified because we could calculate the limit $\tilde{B}(x,+\infty)=\lim_{\lambda\to+\infty}\tilde{B}(x,\lambda)$ and the condition \eqref{assump1} guarantees that all the leading principal minors of $\tilde{B}(x,+\infty)$ are uniformly positive over $\overline{\O}$. Now we give the details of this:

 By \eqref{B1}
\begin{equation} \label{special1}
B_j=\frac 12\(
a^ja^l_{x_j}d_{x_l} + a^l a_{x_l}^j
d_{x_j} \)_{1\leq l\leq n}
+ \left((a^j)^2d_{x_jx_j}-\frac 12 \sum_{k=1}^n a^k a^j_{x_k}
d_{x_k} \right)e_j
\end{equation}
Making use of \eqref{newproperties3},\eqref{newproperties1} and \eqref{newproperties2}, we deduce
\begin{equation}\label{newasy1}
   \lim_{\lambda\to+\infty}\frac{B_j}{d_{x_jx_j}}=(a^j)^2e_j \qquad \text{uniformly for } x\in\overline{\Omega}.
\end{equation}
In the same spirit, for $1\leq i\leq n$ with $i\neq j$, we have
\begin{equation}\label{generalrow}
B_i=\frac 12\(
a^ia^l_{x_i}d_{x_l} + a^l a_{x_l}^i
d_{x_i} \)_{1\leq l\leq n}
+ \left((a^i)^2d_{x_ix_i}-\frac 12 \sum_{k=1}^n a^k a^i_{x_k}
d_{x_k} \right)e_i
\end{equation}
 One could verify by using \eqref{newproperties1} and \eqref{newproperties2} that
 \begin{equation}\label{newasy2}
  \lim_{\lambda\to+\infty}\frac{ B_i}{d_{x_j}}=\frac 12
a^ia^j_{x_i}e_j -\frac 12 a^j a^i_{x_j}
e_i   \qquad \text{uniformly for any } x\in\overline{\Omega}.
\end{equation}
By \eqref{newasy1} and \eqref{newasy2}, we deduce that

 \begin{equation}
 \lim_{\lambda\to+\infty}
    \left(
\begin{array}{cc}
\frac{B_1}{d_{x_j}}\\
\vdots\\
\frac{ B_{j-1}}{d_{x_j}}\\
\frac{ B_j}{d_{x_jx_j}}\\
\frac{ B_{j+1}}{d_{x_j}}\\
\vdots\\
\frac{ B_n}{d_{x_j}}
\end{array}
\right)=\begin{pmatrix}
 -\frac 12 a^ja^1_{x_j} &   \cdots  &   0    &   \frac 12 a^1a^j_{x_1}    & 0  & \cdots    &0\\
 \vdots & \ddots    &   \vdots    &  \vdots     & \vdots  &     &\vdots\\
 0 &  \cdots   &    -\frac 12 a^ja^{j-1}_{x_j}    &  \frac 12a^{j-1}a^j_{x_{j-1}}    & 0  &\cdots    &0\\
 \\
 0 & \cdots    &  0     &  (a^j)^2      & 0 & \cdots    &0\\
  \\
 0 &  \cdots   &  0     &  \frac 12 a^{j+1}a^j_{x_{j+1}}     &  -\frac 12 a^ja^{j+1}_{x_j} &  \cdots   & 0\\
   \vdots &     & \vdots      &   \vdots    &  \vdots &  \ddots    &\vdots\\
 0 &  \cdots   & 0      &   \frac 12 a^na^j_{x_n}    & 0  &  \cdots   &-\frac 12 a^ja^{n}_{x_j}
\end{pmatrix}
 \end{equation}
 uniformly for $x\in\overline{\Omega}$. We deduce from the above formula and \eqref{newproperties3}, \eqref{assump1} that all the leading principal minors of $\tilde{B}(x,\lambda)$ are uniformly positive with a large $\lambda$.  This complete the proof.
\end{proof}
 \begin{proof}[Proof of Theorem \ref{maintheorem}, case 2]
  Here we discuss the case when
  \begin{equation}\label{assump2}
    a^i_{x_j}>0,\quad \text{uniformly over } x\in\overline{\Omega}, 1\leq i\leq n,i\neq j,
  \end{equation}
  where $j$ is a fixed index. In this case, the proof is quite similar as above:
we define a function  $$\ds
d\triangleq d(x)=e^{-\l(x_j-c)}+\sum_{1\leq i\leq n,i\neq j} e^{-\l x_i},\;\; x\in \overline{\Omega},$$ where
$c>0$ satisfies that
\begin{equation}\label{casenewas1}
\max_{x\in\overline\O}\{x_j-c\}+ 1\leq
\min_{x\in
\overline\O}\sum_{1\leq i\leq n,i\neq j}^n|x_i|,
\end{equation}
and
$\l>0$ is a large number will be determined later. Using \eqref{casenewas1}, one could also check that the function $d(x)$ enjoys the following properties:
\begin{itemize}
\item For any $1\leq i\leq n$,
\begin{equation}\label{2newproperties3}
  d_{x_i}<0,~d_{x_{ii}}>0,\qquad \text{uniformly for  $x\in\overline{\Omega}$ }.
  \end{equation}
   \item For any $1\leq i\leq n$,
\begin{equation}\label{2newproperties1}
  \lim_{\lambda\to+\infty}\frac{d_{x_i}}{d_{x_jx_j}}=0\qquad \text{uniformly for  $x\in\overline{\Omega}$ }.
  \end{equation}
  \item For any $1\leq i\leq n$ with $i\neq j$,
\begin{equation}\label{2newproperties2}
  \lim_{\lambda\to+\infty}\frac{d_{x_i}}{d_{x_j}}=0,\quad\lim_{\lambda\to+
  \infty}\frac{d_{x_ix_i}}{d_{x_j}}=0,\quad \text{uniformly for  $x\in\overline{\Omega}$ }.
\end{equation}
 \end{itemize}

  As before, we deduce from \eqref{2newproperties3},\eqref{2newproperties1} and \eqref{2newproperties2} that  the matrix $B$ is uniformly positive definite if and only if all the leading principal minors of
 the matrix
 $\hat{B}(x,\lambda):= \left(
\begin{array}{cc}
-\frac{B_1}{d_{x_j}}\\
\vdots\\
-\frac{ B_{j-1}}{d_{x_j}}\\
\frac{ B_j}{d_{x_jx_j}}\\
-\frac{ B_{j+1}}{d_{x_j}}\\
\vdots\\
-\frac{ B_n}{d_{x_j}}
\end{array}
\right)$
 is uniformly positive over $\overline{\O}$ when $\lambda$ is large enough. By \eqref{B1}
\begin{equation} \label{special1}
B_j=\frac 12\(
a^ja^l_{x_j}d_{x_l} + a^l a_{x_l}^j
d_{x_j} \)_{1\leq l\leq n}
+ \left((a^j)^2d_{x_jx_j}-\frac 12 \sum_{k=1}^n a^k a^j_{x_k}
d_{x_k} \right)e_j
\end{equation}
Making use of \eqref{2newproperties1} and \eqref{2newproperties2}, we deduce
\begin{equation}\label{newasy1}
   \lim_{\lambda\to+\infty}\frac{B_j}{d_{x_jx_j}}=(a^j)^2e_j \qquad \text{uniformly for } x\in\overline{\Omega}.
\end{equation}
In the same spirit, for $1\leq i\leq n$ with $i\neq j$, we have
\begin{equation}\label{generalrow}
B_i=\frac 12\(
a^ia^l_{x_i}d_{x_l} + a^l a_{x_l}^i
d_{x_i} \)_{1\leq l\leq n}
+ \left((a^i)^2d_{x_ix_i}-\frac 12 \sum_{k=1}^n a^k a^i_{x_k}
d_{x_k} \right)e_i
\end{equation}
 One could verify by using \eqref{newproperties1} and \eqref{newproperties2} that
 \begin{equation}\label{newasy2}
  \lim_{\lambda\to+\infty}\frac{ B_i}{d_{x_j}}=\frac 12
a^ia^j_{x_i}e_j -\frac 12 a^j a^i_{x_j}
e_i   \qquad \text{uniformly for any } x\in\overline{\Omega}.
\end{equation}
By \eqref{newasy1} and \eqref{newasy2}, we deduce that

 \begin{equation}
 \lim_{\lambda\to+\infty}
   \left(
\begin{array}{cc}
\frac{B_1}{d_{x_j}}\\
\vdots\\
\frac{ B_{j-1}}{d_{x_j}}\\
\frac{ B_j}{d_{x_jx_j}}\\
\frac{ B_{j+1}}{d_{x_j}}\\
\vdots\\
\frac{ B_n}{d_{x_j}}
\end{array}
\right)=\begin{pmatrix}
 -\frac 12 a^ja^1_{x_j} &   \cdots  &   0    &   \frac 12 a^1a^j_{x_1}    & 0  & \cdots    &0\\
 \vdots & \ddots    &   \vdots    &  \vdots     & \vdots  &     &\vdots\\
 0 &  \cdots   &    -\frac 12 a^ja^{j-1}_{x_j}    &  \frac 12a^{j-1}a^j_{x_{j-1}}    & 0  &\cdots    &0\\
 \\
 0 & \cdots    &  0     &  (a^j)^2      & 0 & \cdots    &0\\
  \\
 0 &  \cdots   &  0     &  \frac 12 a^{j+1}a^j_{x_{j+1}}     &  -\frac 12 a^ja^{j+1}_{x_j} &  \cdots   & 0\\
   \vdots &     & \vdots      &   \vdots    &  \vdots &  \ddots    &\vdots\\
 0 &  \cdots   & 0      &   \frac 12 a^na^j_{x_n}    & 0  &  \cdots   &-\frac 12 a^ja^{n}_{x_j}
\end{pmatrix}
 \end{equation}
 uniformly for $x\in\overline{\Omega}$. The above formula implies

 \begin{equation}
 \lim_{\lambda\to+\infty}
   \left(
\begin{array}{cc}
-\frac{B_1}{d_{x_j}}\\
\vdots\\
-\frac{ B_{j-1}}{d_{x_j}}\\
\frac{ B_j}{d_{x_jx_j}}\\
-\frac{ B_{j+1}}{d_{x_j}}\\
\vdots\\
-\frac{ B_n}{d_{x_j}}
\end{array}
\right)=\begin{pmatrix}
 \frac 12 a^ja^1_{x_j} &   \cdots  &   0    &   -\frac 12 a^1a^j_{x_1}    & 0  & \cdots    &0\\
 \vdots & \ddots    &   \vdots    &  \vdots     & \vdots  &     &\vdots\\
 0 &  \cdots   &    \frac 12 a^ja^{j-1}_{x_j}    &  -\frac 12a^{j-1}a^j_{x_{j-1}}    & 0  &\cdots    &0\\
 \\
 0 & \cdots    &  0     &  (a^j)^2      & 0 & \cdots    &0\\
  \\
 0 &  \cdots   &  0     &  -\frac 12 a^{j+1}a^j_{x_{j+1}}     &  \frac 12 a^ja^{j+1}_{x_j} &  \cdots   & 0\\
   \vdots &     & \vdots      &   \vdots    &  \vdots &  \ddots    &\vdots\\
 0 &  \cdots   & 0      &   -\frac 12 a^na^j_{x_n}    & 0  &  \cdots   &\frac 12 a^ja^{n}_{x_j}
\end{pmatrix}
 \end{equation}
 uniformly for $x\in\overline{\Omega}$.  We deduce from the above formula and \eqref{2newproperties3}, \eqref{assump2} that all the leading principal minors of $\hat{B}(x,\lambda)$ are uniformly positive with a large $\lambda$.  This complete the proof.
\end{proof}

\section{Examples and Comments}

There have been a lot of conditions to ensure the existence of the function $d$.  In \cite{Yao1} (see also \cite{Yao2}),
the author provides a  sectional curvature condition to guarantee the existence of functions $d$. This condition is that
the sign of the sectional curvature function $k$ for
the Riemannian manifold,  with a metric
$A^{-1}=(a^{ij})^{-1}_{1\leq i,j\leq n}$, is either positive or negative over $\overline{\Omega}$.

In this section, we will compare the condition in Theorem~\ref{th1} with the above-mentioned  condition given in \cite{Yao1}.
Then, we will
see  some  advantage can be taken from the condition in Theorem~\ref{th1}.
First, \cite{Yao1} needs the $C^\infty$-regularity for coefficients $a^{i,j}$;
while our Theorem~\ref{th1} only needs the $C^1$-regularity for coefficients.
Second (more important), there are many cases which can be solved by our Theorem~\ref{th1},
but cannot be solved by the sectional curvature condition provided  in \cite{Yao1}. Here, we present an example to explain
the second advantage above-mentioned.

\begin{example}
Let $A=\diag(a^1,a^2)$, where $a^1,a^2\in
C^\infty(\overline\O)$. Suppose that  $
a^2_{x_1}<0$ over $\overline{\Omega}$.  By Theorem \ref{maintheorem} or Corollary \ref{th2}, there is a function $d\in C^2(\overline{\Omega})$ verifying
Condition \ref{con}. However, by making use of the sectional curvature condition provided  in \cite{Yao1}, we cannot imply the existence of the above-mentioned $d$.
In fact, after some  computation, one
can see that  the sectional curvature given by the
metric $A^{-1}$ is as follows:
\begin{equation}\label{wang3.1}
k=\frac{1}{4(a^1a^2)^2}\[ a^2 a^1_{x_1}a_{x_1}^2
 + a^1(a_{x_1}^2)^2   - 2a^1 a^2
  a^2_{x_1x_1}
\].
\end{equation}
From (\ref{wang3.1}), one can  construct many such $a^i$, $i=1,2$, with the property that  $ a^2_{x_1}<0$ over $\overline{\Omega}$, such that the corresponding $k$ changes its sign over $\overline{\Omega}$.

 Here, we provide one of them as follows:
Let $\O=\{(x_1,x_2):\,(x_1-2)^2 + x_2^2 <
3/2\}\subset \dbR^2$. Let $a^1 = e^{\mu_1 x_1}$
and $a^2 = e^{-\mu_2 x_1^2}$, where $\mu_1$ and $\mu_2$ satisfy
\begin{equation}\label{W3.2}
\mu_1>0; \; \mu_2>0;\;  \mu_1 + 2\mu_2 <
2;\; 3\mu_1+18\mu_2>2.
\end{equation}
Clearly, (\ref{W3.2}) has solutions.

 In this case, it is
clear that $ a^2_{x_1}<0$ over
$\overline{\Omega}$ because $x_1>0$ over $\overline{\O}$.
 From (\ref {wang3.1}), we see
$$
\begin{array}{ll}\ds
4(a^1a^2)^2k&\ds=-2\mu_1\mu_2x_1 e^{ \mu_1
x_1-2\mu_2 x_1^2 } + 4\mu_2^2 x_1^2 e^{ \mu_1
x_1-2\mu_2 x_1^2 } + 4\mu_2  e^{ \mu_1
x_1-2\mu_2 x_1^2 } - 8\mu_2^2 x_1^2 e^{ \mu_1
x_1-2\mu_2 x_1^2 }\\
\ns&\ds = -2\mu_2e^{ \mu_1 x_1-2\mu_2 x_1^2
}\big( \mu_1x_1 + 2 \mu_2 x_1^2- 2  \big).
\end{array}
$$
From (\ref{W3.2}), it follows that
$$
(\mu_1x_1 + 2 \mu_2 x_1^2- 2)\big|_{x_1=1}  <0
$$
and
$$
(\mu_1x_1 + 2 \mu_2 x_1^2- 2  )\big|_{x_1=3}>0.
$$
 Hence, $k>0$ in the set $\O\cap
\{(x_1,x_2):\, x_1=1\}$; while  $k<0$ in the set
$\O\cap \{(x_1,x_2):\, x_1=3\}$. From these, we conclude that $k$ changes its sign over
$\overline{\Omega}$. Therefore, the method in
\cite{Yao1} does not work for the current case.

\end{example}

The next two examples are taken from \cite{Yao1} for which the existence can be ensured by either the sectional curvature condition provided in \cite{Yao1}
or our Theorem~\ref{maintheorem}.

\begin{example}
Let  $A=(a^{ij})_{1\leq i,j\leq 2} = \diag(e^{x^3 +
y^3},e^{x^3 + y^3})$. One can directly check that
$$a^2_{x_1}=3y^2 e^{x^3 + y^3}>0.$$
Then, according to Theorem \ref{th2},  there is a $d$
satisfiing \eqref{con1} and \eqref{con2}.
\end{example}

\begin{example}
Let  $A=(a^{ij})_{1\leq i,j\leq 2} = \diag(e^{x + y
},e^{x  + y })$. One can easily check that
$a^2_{x_1}= e^{x + y}>0$. Then, by Theorem
\ref{th2},  there exists a $d$ satisfing
\eqref{con1} and \eqref{con2}.
\end{example}

\begin{remark}
The sectional curvature condition provided in \cite{Yao1} works better than our Theorem~\ref{th1} when $a^{ij}$ is not of diagonal form.
For instance, the Example 3.2 in \cite{Yao1}.
\end{remark}

\section*{Acknowledgements}
The author want to thank Professor Gengsheng Wang for pointing out the prototype of this result and his help on the examples.



\begin{thebibliography}{79}

\bibitem{BLR} C.~Bardos, G.~Lebeau and J. ~Rauch, \it
Sharp sufficient conditions for the observation,
control and stabilization of waves from the
boundary, \sl SIAM J. Control Optim., \rm{\bf
30}(1992), 1024--1065.


\bibitem{FYZ} X.~Fu, J.~Yong and X.~Zhang, \it Exact
controllability for multidimensional semilinear
hyperbolic equations, \sl SIAM J. Control
Optim., \rm{\bf 46}(2007), 1578--1614.


\bibitem{Yao1} P.~F.~Yao, \it On the observability inequalities for exact controllability of wave equations with
variable coefficients, \sl SIAM J. Control
Optim., \rm {\bf 37}(1999), 1568--1599.


\bibitem{Yao2} P.~F.~Yao, \sl Modeling and Control in
Vibrational and Structural Dynamics: A
Differential Geometric Approach, \rm Chapman and
Hall/CRC Press, Boca Raton, 2011.


\end{thebibliography}
\end{document}